\documentclass[10pt,leqno]{amsart}

\usepackage[numbers,sort&compress]{natbib}
\usepackage{mathrsfs}
\usepackage{tikz}
\usepackage{enumerate}
\usepackage{enumitem}
\usepackage[normalem]{ulem}
\usepackage{amssymb,amsthm,amsmath}
\usepackage{mathtools}
\RequirePackage{amsmath}
\RequirePackage{amsthm}
\usepackage{dsfont}
\newcommand{\R}{\mathbb{R}}
\newcommand{\A}{\mathbb{A}}

\newcommand{\N}{\mathbb{N}}
\newcommand{\M}{\mathcal{M}}

\newcommand{\E}{\mathbb{E}}

\newcommand{\bL}{\mathbb{L}}
\newcommand{\bH}{\mathbb{H}}

\newcommand{\mm}{\mathfrak{m}}

\renewcommand{\P}{\mathbb{P}}
\renewcommand{\L}{\mathcal{L}}
\renewcommand{\S}{\mathcal{S}}
\newcommand{\F}{\mathcal{F}}

\newcommand{\X}{\mathcal{X}}

\newcommand{\ce}{\coloneqq}
\newcommand{\ind}[1]{\mathds{1}_{\{#1\}}}

\newcommand{\cobar}{\overline{\text{co}}}

\newcommand{\intt}{\int_0^t}
\newcommand{\intT}{\int_0^T}

\newcommand{\norm}[1]{\left\lVert#1\right\rVert}

\newcommand{\abs}[1]{\left|#1\right|}
\newcommand{\pa}[1]{\left(#1\right)}
\newcommand{\cbra}[1]{\left\{#1\right\}}
\newcommand{\sqbra}[1]{\left[#1\right]}

\newcommand{\nto}{\stackrel{n\to \infty}{\longrightarrow}}

\newcommand{\limn}{\lim_{n\to\infty}}

\newcommand{\calP}{\mathcal{P}}
\newcommand{\calM}{\mathcal{M}}
\newcommand{\calW}{\mathcal{W}}

\newcommand{\calL}{\mathcal{L}}

\newcommand{\calK}{\mathcal{K}}

\DeclareOldFontCommand{\bf}{\normalfont\bfseries}{\mathbf}
\DeclareOldFontCommand{\cal}{\normalfont\bfseries}{\mathcal}

\DeclareMathOperator*{\argmin}{arg\,min}

\DeclareMathOperator*{\as}{\textit{-a.s.}}

\newtheorem{theorem}{Theorem}[section]
\newtheorem{definition}[theorem]{Definition}
\newtheorem{proposition}[theorem]{Proposition}
\newtheorem{corollary}[theorem]{Corollary}
\newtheorem{lemma}[theorem]{Lemma}
\newtheorem{assumption}[theorem]{Assumption}
\newtheorem{remark}[theorem]{Remark}

\begin{document}
\title[MFGs with common noise via Malliavin Calculus]{Mean field games with common noise via Malliavin calculus}
\thanks{Both authors acknowledge financial support by NSF CAREER award DMS-2143861 and the AMS Claytor-Gilmer fellowship.}
\author{Ludovic Tangpi}
\author{Shichun Wang}
\address{Princeton University: Department of Operations Research and Financial Engineering and Program in Computational and Applied Mathematics}
\email{ludovic.tangpi@princeton.edu}
\email{shichun.wang@princeton.edu}

\maketitle

\begin{abstract}
We present a simpler proof of the existence of equilibria for a class of mean field games with common noise, where players interact through the conditional law given the current value of the common noise rather than its entire path. By extending a compactness criterion for Malliavin-differentiable random variables to processes, we establish existence of strong equilibria, where the conditional law and optimal control are adapted to the common noise filtration and defined on the original probability space. Notably, our approach only requires measurability of the drift and cost functionals with respect to the state variable.
\end{abstract}

\section{Introduction}

Mean field games (MFGs), initiated in \cite{LasryLions1}, become challenging to analyze when common noise is involved. The primary difficulty arises from the lack of compactness in the space of measure-valued processes.
A workaround proposed by \cite{CarmonaCommonNoise} addresses this issue by discretizing the common noise in both space and time, thereby compactifying the domain. However, the solutions obtained through this method are so-called \emph{weak equilibria}: they are not necessarily adapted to the Brownian filtration. These equilibria are also considered weak in the probabilistic sense, as the probability setup becomes part of the solution itself. Some extensions have been explored in \cite{Barrasso-Touzi22,TangpiWang23}. 
On the other hand, strong equilibria  (adapted to the common noise filtration) can be obtained through the ``master equation'', a system of stochastic PDEs, or forward-backward SDEs \cite{CardaliaguetDelarueLasryLions19, ahuja2016, Mou-Zhang22,DisplMonotone23}. However, these methods require additional monotonicity and regularity conditions, which are often so stringent that they lead to a unique equilibrium.

In this paper, we propose a markedly simpler strategy to construct strong mean field equilibria (MFE) for a specific class of common noise MFGs, where players interact through the conditional law given the current value of the common noise, rather than the whole path. Our approach leverages a compactness criterion for Sobolev spaces of Wiener functionals developed by \cite{Malliavin1992compact}. This compactness property is akin to Ascoli's condition and is typically verified by means of Malliavin calculus. Notably, this criterion allows us to bypass the discretization procedure and directly establish the existence of equilibria that are adapted to the Brownian filtration (i.e. strong equilibria). Moreover, because our approach does not impose any monotonicity condition on the inputs, the resulting equilibrium is not necessarily unique.

Another advantage of our method is the relaxation of regularity requirements: the drift function of the state variable and the cost functionals only need to be measurable in space. We first prove the existence and uniqueness of strong MFE under the weak formulation proposed in \cite{CarmonaLacker15}, and then bring the equilibrium back to the original probability space. We achieve this by a mimicking argument originally introduced in \cite{brunick2013mimicking}, along with a key result from \cite{Menoukeu-Meyer-Nilssen13}, which establishes that the unique strong solution to an SDE with bounded and measurable drift is Malliavin differentiable. Additionally, we incorporate a \emph{common state} in the setup, which generalizes the Brownian common noise and naturally arises in applications such as optimal execution models in finance \cite{TangpiWang23, CardaliaguetMFG18}.

\section{Setting and Main Result}

 Consider any probability space $(\Omega, \F, \P)$ that supports independent Brownian motions $W$ and $W^0$ taking values in $\R^{d_I}$ and $\R^{d_C}$ respectively, along with independent, $\R^{d_I} \times \R^{d_C}$-valued random variables $(\xi, \xi^c)$.
 Consider the $\P$-completed filtration $\mathbb{F} = (\F_t)_{t\in[0,T]}$ generated by $(\xi, \xi^c, W,W^0)$ as well as $\mathbb{F}^0 = (\F_t^0)_{t\in [0,T]}$  generated by $W^0$ only. All players observe a \emph{common state} $X^{c}$, which solves the following SDE:
 \begin{equation}\label{eq:common.state.SDE}
 dX^{c}_t = b^c(t, X^{c}_t)dt + \sigma^cdW^0_t,\quad X^c_0 = \xi^c.
 \end{equation}
 Consider the Polish space $\calP_p(\R^{d_I})$ of probability measures with finite $p$-th moment, equipped with the Wasserstein metric 
 \begin{equation*}
     \calW_p(\mu,\nu) \ce \bigg(\inf_{\pi \in \Pi(\mu, \nu)}\Big\{\int_{\R^{d_I}\times \R^{d_I}}|x - y|^p\pi(dx, dy)\Big\}\bigg)^{\frac{1}{p}},
 \end{equation*}
where $\Pi(\mu, \nu) \subset \calP_p(\R^{d_I} \times \R^{d_I})$ denotes the set of couplings with marginals $\mu$ and $\nu$. We also recall the duality formula for $p = 1$:
\begin{equation}\label{eq:KantorovichDuality}
      \calW_1(\mu, \nu)= \sup_{\phi \in \text{Lip}_1(\R^{d_I})} \int_{\R^{d_I}} \phi d(\mu - \nu),
 \end{equation}
 where $\text{Lip}_1(\R^{d})$ denotes the set of $1$-Lipschitz functions $\phi:\R^{d} \to \R$.
 Let $\M$ denote the Polish space of $\mathbb{F}^0$-progressively measurable, $\calP_p(\R^{d_1})$-valued stochastic processes $m$ equipped with metric $d_\calM$ defined by
 \begin{equation}\label{eq:calMmetric}
     d^p_\calM(m, m') \ce \E\bigg[\bigg(\intT \calW^2_p(m_t, m_t')dt\bigg)^{p/2}\bigg].
 \end{equation}
 For $\mu \in \calP_p(\R^{d_I})$, set $|\mu|^p \ce \int_{\R^{d_I}}|x|^p\mu(dx) < \infty$. Let $A \subset \R^{d_A}$ be the action space, and the set of \emph{admissible} strategies $\A$ be the set of $\mathbb{F}$-progressively measurable, square integrable $A$-valued processes on $\Omega \times [0, T]$ such that the following controlled state process has a unique square-integrable strong solution with values in $\R^{d_I}$:
 \begin{equation}\label{eq:Strong_State}
     X^{\alpha}_t = \xi + \intt b(s, X^{\alpha}_s, m_s, \alpha_s)ds +  \sigma W_t + \sigma^0 W^0_t, \quad t \in [0, T],
 \end{equation}
 with given $m \in \M$. We consider the following objective for running and terminal cost functionals $f$ and $g$:
 \begin{equation}\label{eq:Strong_Objective}
     J^{m}(\alpha) = \E\sqbra{\intT f(t, X^{\alpha}_t, m_t, \alpha_t)dt + g(X_T, m_T)}.
 \end{equation}
 The coefficients are Borel measurable, deterministic functions
 \begin{gather*}
         b:[0, T] \times \R^{d_I} \times \calP_p(\R^{d_I}) \times A \to \R^{d_I}, \quad b^c: [0, T] \times \R^{d_C} \to \R^{d_C}\\
         (\sigma, \sigma^0, \sigma^c) \in \R^{d_I \times d_I} \times \R^{d_I \times d_C} \times \R^{d_C \times d_C}\\
         f: [0, T] \times \R^{d_I} \times \calP_p(\R^{d_I}) \times A \to \R, \quad g: \R^{d_I} \times \calP_p(\R^{d_I}) \to \R.
 \end{gather*}
We work under the setup where players interact through the conditional law given the \emph{current} value of the common state. 
 \begin{definition}\label{Def:Strong_MFGsolution}
 A mean field equilibrium is a pair $(\hat m,\hat\alpha) \in \calM \times \A$; $\hat\alpha \in \A$ minimizes the objective $J^{\hat{m}}$ defined in \eqref{eq:Strong_Objective} over $\A$. Moreover, for every $t \in [0, T]$, $\hat m_t$ is the conditional law under $\P$ of $X^{\hat\alpha}_t$ given $X^c_t$, which we denote as
 \begin{equation*}
     \hat m_t = \calL^{\P}\big(X^{\hat\alpha}_t | X^c_t\big), \quad \P\as
 \end{equation*} 
 \end{definition}

\begin{remark}
 \label{rem:def.equil}
    The attentive reader has noticed that our definition of equilibrium is a bit different from the usual one where one imposes that the population law  matches the conditional law of the optimal state given the entire past of the common state: $\hat m_t = \calL^{\P}\big(X^{\hat\alpha}_t | \F^{X^c}_t\big), \P\as$
    This definition is more natural and inspired from propagation of chaos in the presence of common noise.
    Ours is a simplification that will make the proof easier.
    
    However, our definition is also consistent with some $N$-particle models.
    For instance, if we consider the simple particle system $
    dX^{i,N}_t = \delta\big(\frac1N\sum_{j=1}^NX^{j,N}_t - X^{i,N}_t\big)dt + dW^i_t + dW^0_t,\,\, X^{i,N}_0 = X^i_0$, then we know by propagation of chaos that the sample average $\bar X_t=\frac1N\sum_{j=1}^NX^{j,N}_t$ converges to $\E[X_t|\mathcal{F}^{W^0}_t]$ where $X$ solves the associated conditional McKean-Vlasov equation, whereas computing $\bar X_t$ directly and using the law of large numbers shows that it also converges to $\E[X_0^1] + W_t^0$, showing that $\E[X_t|\mathcal{F}_{t}^{W^0}] = \E[X_t|W_t^0]$.
    We will further elaborate on our definition in Remark \ref{remark:generalMFG} below.
\end{remark}

 Define the reduced Hamiltonian $H: [0, T] \times \R^{d_I} \times \calP_p(\R^{d_I}) \times A \times \R^{d_I} \to \R$ and minimized Hamiltonian $h$ respectively by
 \begin{gather*}
     H(t, x, \mu, a, z) \ce f(t, x, \mu, a) + z \cdot \sigma^{-1}b(t, x, \mu, a)\\
     h(t, x, \mu, z) \ce \inf_{a \in A} H(t, x, \mu, a, z).
 \end{gather*}
\begin{assumption}\label{Assumption:E}~
There exists $p \geq 2$ such that $\xi \in L^{2p}(\R^{d_I})$, $\xi^c \in L^{2p}(\R^{d_C})$ and:
 \begin{enumerate}[label=(\ref{Assumption:E}.\arabic*)]
         \item \label{Assumption:E1} For all $(t, x)\in [0, T] \times \R^{d_I}$, the map $b(t, x,\cdot, \cdot)$ is continuous. The functions $b, b^c$ are uniformly bounded. The matrices $\sigma$ and $\sigma^c$ are invertible. 
     \item \label{Assumption:E2} For all $(t, x)\in [0, T] \times \R^{d_I}$, the maps $f(t, x, \cdot, \cdot)$ and $g(x, \cdot)$ are continuous, and there exists $C > 0$ such that for all $(t, x, \mu, a) \in [0, T] \times \R^{d_I} \times \calP_p(\R^{d_I}) \times A:$ 
     \begin{equation*}
         |g(x, \mu)| + |f(t, x, \mu, a)| \leq C(1+ |x|^p + |\mu|^p  + |a|^p).
     \end{equation*}
     \item \label{Assumption:E3} The action space $A \subset \R^{d_A}$ is compact and convex. Moreover, for all $(t, x, \mu, z) \in [0, T] \times \R^{d_I} \times \calP_p(\R^{d_I}) \times \R^{d_I}$, there exists a unique minimizer of $H(t, x, \mu, \cdot, z)$.
 \end{enumerate}
\end{assumption}

\noindent  We stress that there is no regularity requirement in $x$ imposed on $f, g$, and more importantly, $b$. Under Assumption \ref{Assumption:E1}, functions $h$ and $H$ are both Lipschitz in $z$. Under Assumption \ref{Assumption:E3}, we define
 $\hat{a}(t, x, \mu, z) \ce \argmin_{a \in A} H(t, x, \mu, a, z)$. 
 \begin{assumption}\label{Assumption:M}~
The following set is convex for each $(t, x, \mu) \in [0, T] \times \R^{d_I} \times \calP_p(\R^{d_I})$:
     $$\Big\{\big(b(t, x, \mu, a), z\big): z \in \R, a \in A, z \geq f(t, x, \mu, a)\Big\}.$$
 \end{assumption}
 
This convexity assumption is common for measurable selection arguments \cite{haussmann2006existence, lacker2015mean}. A typical example where it is satisfied is where the drift $b$ is affine in $a$, and $f$ is convex in $a$. We now state the main result on the existence of equilibrium.
 \begin{theorem}\label{Theorem:Main}
 \label{thm:main}
     Under Assumptions \ref{Assumption:E} and \ref{Assumption:M}, the mean field game admits a mean field equilibrium $(\hat{m}, \hat{\alpha}) \in \calM \times \A$. Moreover, its optimal control is Markovian. Specifically, there exists a measurable function $\hat{\alpha}_M: [0, T] \times \R^{d_I} \times \R^{d_C} \to A$ such that $$\hat{\alpha}_t = \hat{\alpha}_M(t, X^{\hat{\alpha}}_t, X^{c}_t), \quad \P\otimes dt \as$$
     where $(X^{\hat{\alpha}}, X^{c})$ uniquely solves the SDE
     $$\begin{cases}dX^{\hat{\alpha}}_t = b(t, X^{\hat{\alpha}}_t, \hat{m}_t, \hat{\alpha}_M(t, X^{\hat{\alpha}}_t, X^{c}_t))dt + \sigma dW_t + \sigma^0 dW^0_t, \quad X^{\hat{\alpha}}_0 = \xi.\\
     dX^c_t = b^c(t, X^c_t)dt + \sigma^cdW^0_t, \quad X^c_0 = \xi^c.
     \end{cases}$$
 \end{theorem}
 We emphasize the fact that the equilibrium is strong in the probabilistic sense, where the probability setup is not part of the solution. The equilibrium is also strong in the control-theoretic sense, where the filtrations are the $\P$-complete natural filtrations of the Brownian motions, and $\hat{m}$ (resp. $\hat\alpha$) is $\mathbb{F}^0$ (resp. $\mathbb F$) adapted.

 \section{Existence of MFE: Proof of Theorem \ref{thm:main}}\label{sec:proof}
We first reformulate the MFG in the weak sense, proving the existence and uniqueness of the equilibrium in that framework, before returning to the original setting described above.

 \subsection{Weak Formulation of MFG}\label{subsection:weakformulation}
 Let us begin by introducing some notation of spaces and norms that will be used throughout the proof. Consider the probability space $(\Omega, \F, \widetilde{\P})$ with probability measure $\widetilde{\P}$ possibly different from $\mathbb{P}$. For $k\in \mathbb{N}$ and a filtration $\mathbb{G}$ we denote by $\mathbb{S}^2(\R^k,\mathbb{G}; \widetilde{\P})$ the set of continuous, $\mathbb{G}$-adapted and $\R^k$-valued processes $Y$ such that $\|Y\|^2_{\mathbb{S}^2(\R^k,\mathbb{G}; \widetilde{\P})}:=\E^{\widetilde{\P}}[\sup_{t\in [0,T]}|Y_t|^2]<\infty$.
 We denote by $\mathbb{H}^2(\R^k,\mathbb{G}; \widetilde{\P})$ the space of $\mathbb{G}$-progressively measurable, $\mathbb{R}^k$-valued processes $Z$ such that $\|Z\|^2_{\mathbb{H}^2(\R^k,\mathbb{G}; \widetilde{\P})}\ce\E^{\widetilde{\P}}[\int_0^T|Z_t|^2dt]<\infty$. We drop the probability measure from the notation when $\widetilde \P = \P$.

Consider the driftless and uncontrolled state process:
 \begin{equation}\label{eq:driftlessState}
     X_t = \xi +  \sigma W_t + \sigma^0 W^0_t.
 \end{equation}
 Let $\A_w$ denote the set of admissible strategies in the weak formulation, which are simply $A$-valued, $\mathbb{F}$-progressive processes. For a fixed $\alpha \in \A_w$ and $m \in \M$, define the equivalent probability measure $\P^{\alpha, m}$ as well as the process $W^{\alpha, m}$ by
 \begin{equation}\label{eq:Girsanov}
     \frac{d\P^{\alpha, m}}{d\P} \ce \mathcal{E}\pa{\int_0^\cdot \sigma^{-1}b(s, X_s, m_s, \alpha_s)\,dW_s}_T, \quad W^{\alpha,m}_t \ce W_t - \intt  \sigma^{-1}b(s, X_s, m_s, \alpha_s)ds
 \end{equation}
 where $\mathcal{E}(\cdot)$ denotes the stochastic exponential. Assumption \ref{Assumption:E1} ensures that $\P^{\alpha, m}$ is well-defined and Girsanov's theorem applies, so $(W^{\alpha,m}, W^0)$ is a Brownian motion under $\P^{\alpha, m}$. Moreover, under $\P^{\alpha, m}$, $X$ follows the dynamics
 \begin{equation*}
     dX_t = b(t, X_t, m_t,  \alpha_t)dt + \sigma dW^{\alpha, m}_t + \sigma^0dW^0_t, \quad X_0 = \xi.
  \end{equation*} 
 The objective under the weak formulation reads
 \begin{equation}\label{eq:Objective}
     J^{m}_{w}(\alpha) \ce \E^{\P^{\alpha, m}}\sqbra{\intT f(t, X_t, m_t, \alpha_t)dt + g(X_T, m_T)}.
 \end{equation}

 \begin{definition}\label{Def:MFGsolution}
 A mean field equilibrium under the weak formulation is a pair $(\hat m,\hat\alpha) \in \calM \times \A_w$ where $\hat\alpha$ minimizes the objective $J^{\hat{m}}_{w}$ defined in \eqref{eq:Objective} over $\A_w$. Moreover, for almost every $t$: $$\hat m_t = \calL^{\P^{\hat{\alpha}, \hat{m}}}\big(X_t | X^c_t), \quad \P\as$$
 \end{definition}

Note that the probability space is part of the solution. However, the equilibrium here is still strong since the filtrations are Brownian, and $\hat{m}$ is $\mathbb{F}^0$-adapted.

 \begin{theorem}\label{Theorem:MainWeak}
     Under Assumption \ref{Assumption:E} and Assumption \ref{Assumption:M}, the mean field game admits an equilibrium under the weak formulation.
 \end{theorem}

 \subsection{Malliavin Calculus}
 We refer to \cite[Chapter 1]{nualart2006malliavin} for definitions related to Malliavin calculus. Specifically, for $p \geq 1$, the Sobolev space $\mathbb{D}^{1,p}(\R^d)$ of Malliavin differentiable, $\R^d$-valued random variables $\xi$ is equipped with the norm 
\begin{equation*}
    \|\xi\|^p_{1,p} := \E\sqbra{|\xi|^p} + \E\sqbra{\abs{\int_0^T|D_t\xi|^2dt}^{p/2}},
\end{equation*}
where $D_t\xi$ denotes the Malliavin derivative of $\xi$ with respect to $(W, W^0)$ at time $t$. For a given filtration $\mathbb{G}$, the space of Malliavin differentiable processes $\mathbb{L}_{a}^{1,p}(\R^d; \mathbb{G})$ is the space of $\R^d$-valued, $\mathbb{G}$-adapted processes $X$ such that for almost every $t\in [0,T]$, $X_t \in \mathbb{D}^{1,p}(\R^d)$;  the $(d_I + d_C) \times d$-dimensional process $(D_sX_t)_{s\in [0,T]}$ admits a square integrable progressively measurable version and 
\begin{equation*}
    \|X\|^p_{\mathbb L^{1,p}_a(\R^k; \mathbb{G})}\ce \E\bigg[\bigg(\intT|X_t|^2dt\bigg)^{p/2}\bigg] + \E\bigg[\pa{\int_0^T\int_0^T|D_sX_t|_F^2dsdt}^{p/2}\bigg]<\infty,
\end{equation*}
where the $|\cdot|_F$ is the Frobenius norm. We borrow these tools mainly for the following compactness criterion due to \cite{Malliavin1992compact}. See also \cite[Corollary 5.3]{Menoukeu-Meyer-Nilssen13}.

 \begin{proposition}\label{Prop:MalliavinCompactVariable}
     A sequence of $\F^0_T$-measurable random variables $(\xi^n)_{n \in \N} \in \mathbb{D}^{1, 2}(\R^k)$ is relatively compact in $L^2(\Omega)$ if it is uniformly bounded in $L^2(\Omega)$, and there exist $\delta > 0, C > 0$ such that for all $n \in \N$ and $0 \leq s \leq t \leq T$
     \begin{equation*}
         \E[\norm{D_t\xi^n - D_{s}\xi^n}^2] \leq C|t - s|^{\delta} \text{ and } \sup_{0 \leq u \leq T}\E[\norm{D_u\xi^n}^2] \leq C.
     \end{equation*}
 \end{proposition}

By leveraging this compactness criterion, the authors of \cite{Menoukeu-Meyer-Nilssen13} demonstrate that the unique solution to SDEs with bounded measurable drifts and constant diffusions is Malliavin differentiable. We summarize their primary results in the following proposition, which serves as a crucial component in our proof.
 \begin{definition}\label{Def:calS}
     For fixed $C > 0$ and $\delta > 0$, define $\S_{C, \delta}(\R^{k})$ to be the subset of $\mathbb L_{1,2}^a(\R^{k}; \mathbb{F})$ whose elements are stochastic processes $M$ satisfying $\E\sqbra{|M_t|^2}\le C$ for all $t \in [0, T]$ and
     \begin{equation*}
     \E[\norm{D_sM_t - D_{s'}M_t}^2] \leq C|s - s'|^{\delta}, \text{ and } \sup_{0 \leq u \leq t} \E[\norm{D_uM_t}^2]\leq C.
     \end{equation*}    
     for all $0 \leq s, s' \leq t \leq T$.
 \end{definition}
 \begin{proposition}\label{proposition:MalliavinDiffBounded}
     For uniformly bounded, Borel-measurable functions $\tilde{b}: [0, T] \times \R^{d_I} \times \R^{d_C} \to  \R^{d_I}$, the following SDE has a unique strong solution: 
     \begin{equation}\label{label:systemSDE}
         \begin{cases}
         d\widetilde{X}_t = \tilde{b}(t, \widetilde{X}_t, X^c_t)dt + \sigma dW_t + \sigma^0 dW^0_t, \quad \widetilde{X}_0 = \xi\\
         dX^c_t= b^c(t, X^c_t)dt + \sigma^c dW^0_t, \quad X^c_0 = \xi^c.
     \end{cases}
     \end{equation}
     The solution $(\widetilde{X}, X^c)$ is in $\bL_{1,2}^a(\R^{d_I} \times \R^{d_C}, \mathbb{F})$. For any sequence of uniformly bounded, smooth drifts $(\tilde{b}_n, b^{c}_n)$ that approximates $(\tilde{b}, b^c)$ a.e. with respect to the Lebesgue measure with associated (unique) solutions $(\widetilde{X}^n, X^{c, n})$, there exists a subsequence $\{n_k\}_{k \in \N}$ such that $(\widetilde{X}^{n_k}, X^{c, n_k})$ converges to $(\widetilde{X}, X^c)$ in $\mathbb{H}^2(\R^{d_I} \times \R^{d_C}, \mathbb{F})$.
     Moreover, there exists $C > 0$ that only depends on $\lVert\tilde{b}\rVert_{\infty}, \lVert b^c\rVert_{\infty}$, and the dimensions $d_I, d_C$ such that $(\widetilde{X}^{n}, X^{c, n}) \in  \S_{C,1}(\R^{d_I}) \times \S_{C, 1}(\R^{d_C})$ for all $n$. 
 \end{proposition}
 \begin{proof}
     First recall that $(\xi, \xi^c)$ are independent of $(W, W^0)$, so the initial conditions do not affect the Malliavin differentiability of the solution. Consider the following SDE:
     $$d\hat{X}_t = \Sigma^{-1}\hat{b}(t, \Sigma \hat{X}_t)dt + dB_t, \quad \hat{X}_0 = (\xi, \xi^c)^\intercal \in L^{2p}(\R^{d_I + d_C})$$
    where for $(t, x, x^c) \in [0, T] \times \R^{d_I} \times \R^{d_C}$:
     $$\Sigma \ce \begin{pmatrix}
         \sigma & \sigma^0\\
         0 & \sigma^c
     \end{pmatrix} \in \R^{d_I + d_C} \times \R^{d_I + d_C}, \quad \hat{b}(t, (x,x^c)^\intercal) \ce \begin{pmatrix}
         \tilde{b}(t, x, x^c)\\
         b^c(t, x^c)
     \end{pmatrix} \in \R^{d_I + d_C}.$$
     By applying \cite[Theorem 3.3]{Menoukeu-Meyer-Nilssen13}, we obtain the unique, Malliavin differentiable solution $\hat{X}$. We can then define $(\widetilde{X}, \widetilde{X}^c) \ce \Sigma\hat{X}$, which can be readily verified to be the unique solution to \eqref{label:systemSDE}. The remaining claims follow directly from the proof of \cite[Theorem 3.3]{Menoukeu-Meyer-Nilssen13}. Specifically, \cite[Lemma 3.5, Corollary 3.6]{Menoukeu-Meyer-Nilssen13} establish the existence of constants $C, \delta > 0$ such that $(\widetilde{X}^{n}, X^{c, n}) \in \S_{C,\delta}(\R^{d_I}) \times \S_{C,\delta}(\R^{d_C})$  for all $n$. In fact, a careful read of the proof of \cite[Lemma 3.5]{Menoukeu-Meyer-Nilssen13} reveals that $\delta  = 1$.
    Furthermore, \cite[Lemma 3.14, Lemma 3.16]{Menoukeu-Meyer-Nilssen13} guarantee the existence of a subsequence $\{(\widetilde{X}^{n_k},  X^{c, n_k})\}_k$ that converges to $(\widetilde{X}, X^c)$ pointwise in $L^2(\Omega)$. By the dominated convergence theorem, this subsequence also converges in $\bH^2(\R^{d_{I}}\times \R^{d_C}, \mathbb{F})$.
 \end{proof}

Define $\X$ to be the set of pairs $(\widetilde{X}, X^c)$ that (uniquely) solve the system of SDEs \eqref{label:systemSDE} for some measurable, bounded $\tilde{b}$ with bound $\lVert\tilde{b}\rVert_\infty \leq \norm{b}_\infty$. By Proposition \ref{Prop:MalliavinCompactVariable}, there exists $C >0$ such that $\X \subset \overline{\S_{C,1}(\R^{d_I})} \times \overline{\S_{C, 1}(\R^{d_C})}$, where the closure is taken in $\mathbb{H}^2(\R^{d_I} \times \R^{d_C}, \mathbb{F})$.
 \begin{corollary}\label{Cor:MalliavinCompact}
The set $\X$ is precompact in $\bH^2(\R^{d_I} \times \R^{d_C}, \mathbb{F})$.
 \end{corollary}
 \begin{proof}
      Take a sequence in $\X$ with first coordinate $\widetilde{X}^n$ corresponding to drifts $\tilde{b}^n$. Denote by $q^n$ the uniformly bounded processes $\tilde{b}^n(\cdot, \widetilde{X}^n_\cdot, X^c_{\cdot})$. Dunford-Pettis theorem implies that along a subsequence, there exists a $\P\otimes dt$-unique $q^\infty$ such that $q^n \nto q^\infty$ weakly in $\bH^2(\R^{d_I}, \mathbb{F})$. Without loss of generality, we assume the whole sequence $q^n$ converges weakly to $q^\infty$. Define $\widetilde{X}^{\infty}$ by $$\widetilde{X}^{\infty}_t \ce \xi + \intt q^\infty_sds + \sigma W_t + \sigma^0W^0_t, \quad t \in [0, T].$$ Then for almost every $t$, the random variables $\widetilde{X}^n_t$ converges weakly to $\widetilde{X}^{\infty}_t$ in $L^2(\Omega)$. Indeed, for an arbitrary $\R^{d_I}$-valued random variable $Z \in L^2(\Omega)$, 
     \begin{align*}
     \limn \E[Z \cdot \widetilde{X}^n_t] & = \limn \E\sqbra{\intT \ind{s \leq t} Z \cdot q^n_sds + Z\cdot (\xi + \sigma W_t + \sigma^0W^0_t)}\\
     & =  \E\sqbra{\intT \ind{s \leq t} Z \cdot q^\infty_sds + Z\cdot (\xi + \sigma W_t + \sigma^0W^0_t)} = \E[Z \cdot \widetilde{X}^{\infty}_t].
     \end{align*}
     Since $\widetilde{X}^n \in \overline{\S_{C,1}(\R^{d_I})}$, the sequence $(\widetilde{X}^n_t)_{n \geq 1}$ is also relatively compact in $L^2(\Omega)$. This yields a subsequence $(\widetilde{X}^{n_k}_t)_k$ that converges (strongly) in $L^2(\Omega)$. By uniqueness of the weak limit $\widetilde{X}^\infty_t$, the whole sequence $\widetilde{X}^n_t$ converges strongly to $\widetilde{X}^{\infty}_t$ in $L^2(\Omega)$. Since $t$ was taken arbitrarily, dominated convergence implies that $\lVert\widetilde{X}^n - \widetilde{X}^{\infty}\rVert_{\bH^2(\R^{d_I}, \mathbb{F})} \nto 0$.
 \end{proof}

 We denote by $\M_{C, 1} \subset \calM$ the set of processes $\cbra{\calL^{\P}(\widetilde{X}_\cdot\mid X^c_\cdot): (\widetilde{X}, X^c) \in \X}$.  Before applying Brouwer's fixed point theorem, we first interpret elements in $\calM$ as processes that take values in the space of finite signed measures $\mu$ over $\R^{d_I}$ with finite first absolute moment. We endow this space with a variant of the Kantorovich-Rubinstein norm, defined as follows:
 \begin{equation*}
     \lVert\mu\rVert_{\text{KR}} \ce |\mu(\R^{d_I})| + \sup_{\phi \in \text{Lip}_1(\R^{d_I})}\int_{\R^{d_I}}\phi(x)d\mu(x).
 \end{equation*}
 Notably, the metric induced by this norm coincides with $\calW_1$ when restricted to $\calP_1$. Define the set $\calK\ce \cobar\pa{\M_{C, 1}}$, the closed convex hull of $\M_{C, 1}$ in $\calM$.

 \begin{lemma}\label{lemma:calK}
     $\calK$ is a convex and compact subset of $\calM$. 
     Moreover, for each $m \in \calK$ there exists a unique, jointly measurable function $\varphi_m: [0, T] \times \R^{d_C} \to \calP_p(\R^{d_I})$ such that for every $t \in [0,T]$ we have $m_t = \varphi_m(t, X^c_t)$, $\P$-almost surely.
 \end{lemma}
 \begin{proof}
    We demonstrate the continuity of the mapping $\X \ni (\widetilde{X}, X) \mapsto \calL^{\P}(\widetilde{X}_\cdot\mid X^c_\cdot) = m \in \M$. Consider a sequence $(\widetilde{X}^n, X^{c})_n$ in $\X$ where $\widetilde{X}^n$ converges to $\widetilde{X}^\infty$ and $X^c$ is fixed. 
    Let $m^n$ and $m^\infty$ denote the corresponding conditional law processes for $\widetilde{X}^n$ and $\widetilde{X}^\infty$, respectively, given $X^c$. In light of the definition of $d_\calM$ in \eqref{eq:calMmetric}, Jensen's inequality gives 
    \begin{equation}\label{eq:dMJensen}
        d_\calM^p(m^n, m^\infty) \leq T^{p/2 - 1}\E\bigg[\intT \calW^p_p(m^n_t, m^\infty_t)dt\bigg],
    \end{equation}
    so it suffices to show that $\E\big[\calW^p_p(m^n_t, m^\infty_t)\big] \nto 0$  for almost every $t \in [0, T]$.
    
    Note that $\P\circ (X^c_t)^{-1}$ is absolutely continuous with respect to Lebesgue measure (by e.g. Girsanov's theorem). By definition of Wasserstein distance, for almost every $x^c \in \R^{d_C}$ and $t \in [0, T]$ we have
    \begin{equation*}
        \calW^p_p\Big(\L^{\P}(\widetilde{X}^n_t|X^c_t = x^c), \L^{\P}(\widetilde{X}^\infty_t|X^c_t = x^c)\Big) \leq \E\Big[|\widetilde{X}^n_t - \widetilde{X}^\infty_t|^p \Big | X^c_t = x^c\Big].
    \end{equation*}
    From the proof of Corollary \ref{Cor:MalliavinCompact}, we know that $\widetilde{X}^n_t$ also converges to $\widetilde{X}^{\infty}_t$ in $L^2(\Omega)$ for almost every $t$. Therefore, 
    \begin{equation*}
        \E\Big[\calW^p_p(m^n_t, m^\infty_t)\Big] \leq \E\Big[\E\big[|\widetilde{X}^n_t - \widetilde{X}^\infty_t|^p \big | X^c_t\big]\Big] \leq \E\Big[|\widetilde{X}^n_t - \widetilde{X}^\infty_t|^p \Big]  \nto 0,
    \end{equation*}
    which shows continuity of the mapping $\X \ni (\widetilde{X}, X) \mapsto \calL^{\P}(\widetilde{X}_\cdot\mid X^c_\cdot) = m \in \M$.
    This implies that $\calM_{C, 1}$ is pre-compact. The closed convex hull of a pre-compact set is compact (e.g \cite[Theorem 5.35]{bookInfDimAnalysis06}) and obviously convex.

     Regarding the representation of $m$ by jointly measurable functions, note that each $(\widetilde{X}, X^c) \in \X$ is progressively measurable. By disintegration theorem, there exists a $\P \circ (X^c)^{-1}$-unique family of measurable functions $\varphi_m: [0, T] \times \R^{d_C} \to \calP_p(\R^{d_I})$ such that $m_t = \varphi_m(t, X^c_t)$ almost surely for every $t$. In other words, elements in $\calM_{C, 1}$ admit this representation, as do elements in $\calK$, since measurability is preserved under convex combinations and closure.
 \end{proof}

 \subsection{Proof of Theorem \ref{thm:main}}
 We establish the existence of MFG equilibria by identifying fixed points of a mapping from $\calK$ to $\M_{C,1}$. To do so, we first prove an intermediate result for the MFG under weak formulation.\\

 \noindent\emph{Proof of Theorem \ref{Theorem:MainWeak}:} We split the proof into three steps.\\
 \emph{Step 1: Optimality via Comparison Principle.} Boundedness of $|\sigma^{-1}b|$ ensures that $H$ and $h$ are Lipschitz in $z$. Let $m \in \calK$ be given. For each $\alpha \in \A_w$, the growth conditions in Assumption \ref{Assumption:E2} implies that the process $H(\cdot, X_\cdot, m_\cdot, \alpha_\cdot, 0)$ belongs to $\mathbb{H}^2(\R,\mathbb F)$, and the random variable $g(X_T, m_T)$ is square integrable.
 Thus, by \cite[Theorem 2.1]{el1997backward} there is a unique solution $(Y^{m,\alpha},Z^{m,\alpha},Z^{0,m,\alpha}) \in \mathbb{S}^2(\R, \mathbb F)\times \mathbb{H}^2(\R^{d_I}, \mathbb F)\times \mathbb{H}^2(\R^{d_C}, \mathbb F)$ to the Lipschitz BSDE
 \begin{align*}
     Y^{m,\alpha}_t = g(X_T, m_T) + \int_t^TH(s&, X_s, m_s, \alpha_s, Z^{m,\alpha}_s)ds\\
     & -\int_t^TZ^{m, \alpha}_sdW_s - \int_t^TZ^{0,m,\alpha}_sdW_s^0.
 \end{align*}
 By Girsanov's theorem, we have
 \begin{equation*}
     Y^{m, \alpha}_t = \E^{\P^{\alpha, m}}\Big[g(X_T, m_T) + \int_t^Tf(s, X_s, m_s, \alpha_s)ds\mid \F_t\Big],
 \end{equation*}
 which implies $J^m_{w}(\alpha) = \E^{\P^{\alpha, m}}[Y^{m,\alpha}_0] = \E[Y^{m,\alpha}_0]$ as the two probability measures $\P^{\alpha, m}$ and $\P$ agree on $\F_0$. By \cite[Theorem 2.1]{el1997backward} again, there exists a unique $(Y^m, Z^m, Z^{0,m}) \in  \mathbb{S}^2(\R, \mathbb F)\times \mathbb{H}^2(\R^{d_I}, \mathbb F)\times \mathbb{H}^2(\R^{d_C}, \mathbb F)$ solving
 \begin{equation}
 \label{eq:optim.BSDE}
     Y^{m}_t = g(X_T, m_T) + \int_t^Th(s, X_s, m_s, Z^{m}_s)ds -\int_t^TZ^{m}_sdW_s - \int_t^TZ^{0,m}_sdW_s^0.
 \end{equation}
Let $\hat\alpha_t = \hat a(t, X_t, m_t, Z_t^{m})$, where $\hat{a}$ is the minimizer function defined after Assumption \ref{Assumption:E}. We have $J^m(\hat\alpha) = \E^{\P^{\hat\alpha, m}}[Y^{m}_0]$. Applying the comparison principle for Lipschitz BSDEs \cite[Theorem 2.2]{el1997backward}, we obtain:
 \begin{equation*}
     J^m_{w}(\hat\alpha) = \E^{\P^{\hat\alpha, m}}[Y^{m}_0] = \E[Y^{m}_0] \le \E[Y^{m, \alpha}_0] = \E^{\P^{\alpha, m}}[Y^{m,\alpha}_0] = J^m_{w}(\alpha).
 \end{equation*}
 Thus, $\hat\alpha$ is optimal for $m$. We now define the function $\Phi$ mapping $\calK$ into $\calM$ by
 \begin{equation*}
     \Phi(m) = \mm \ce (\calL^{\P^{\hat{\alpha}, m}}\big(X_t | X^c_t))_{t\in [0,T]}.
 \end{equation*}
 Finally, it follows directly from Definition \ref{Def:MFGsolution} that fixed points of $\Phi$, together with their corresponding optimal control derived from \eqref{eq:optim.BSDE}, constitute MFG equilibria under the weak formulation.
 
 \noindent \emph{Step 2: Showing $\Phi$ maps $\mathcal{K}$ into $\M_{C, 1}$.}
 
 \begin{lemma}\label{Lemma:Mimicking}
     Let $m \in \calK$ be fixed with corresponding $\varphi_m$ from Lemma \ref{lemma:calK} and let $\alpha \in \A_w$ be any admissible control. Under Assumption \ref{Assumption:M}, there exists a Borel measurable function $\alpha_M:[0, T] \times \R^{d_I} \times \R^{d_C} \to A$ such that the Markovian control $\widetilde{\alpha}$ defined by $\widetilde\alpha_t \ce \alpha_M(t, X_t, X^c_t)$ satisfies $J^m_w(\alpha) \geq J^m_w(\widetilde{\alpha})$. Suppose $\alpha = \hat{\alpha}$ from Step 1. Then $\hat\alpha = \widetilde{\alpha}$, $\P\otimes dt$-almost everywhere.
 \end{lemma}
 \begin{proof}
    Recall that under $\P^{\alpha, m}$, the state process $X$ follows
    \begin{equation*}
        dX_t = b(t, X_t, \varphi_m(t, X^c_t), \alpha_t)dt + \sigma dW^{\alpha, m}_t + \sigma^0 dW^0_t, \quad X_0 = \xi.
    \end{equation*} 
 Assumption \ref{Assumption:M}, along with the continuity of $b, f$ in $a$ and the closedness of $A$, ensures the existence of a measurable function $(t, x, x^c) \mapsto \alpha_M(t, x, x^c)$ satisfying the following:
 \begin{align}
     \label{label:Mimicking1}b(t, X_t, \varphi_m(t, X^c_t), \alpha_M(t, X_t, X^c_t)) & = \E^{\P^{\alpha, m}}\sqbra{b(t, X_t, m_t, \alpha_t) \mid (X_t, X^c_t)}\\
     \label{label:Mimicking2}f(t, X_t, \varphi_m(t, X^c_t), \alpha_M(t, X_t, X^c_t)) & \leq \E^{\P^{\alpha, m}}\sqbra{f(t, X_t, m_t, \alpha_t) \mid (X_t, X^c_t)}
 \end{align}
 almost surely for almost every $t \in [0, T]$ (see e.g. \cite[Theorem A.9]{haussmann1990existence}). 
 By the mimicking theorem \cite[Corollary 3.7]{brunick2013mimicking} and \eqref{label:Mimicking1}, there exists some filtered probability space $(\widetilde{\Omega}, \widetilde{\F}, (\widetilde{\F}_t)_{t \in [0, T]}, \widetilde{\P})$ supporting independent Brownian motions $(\widetilde{W}, \widetilde{W}^0)$, random variables $(\tilde{\xi}, \tilde{\xi^c})$ and a (weak) solution $(\widetilde{X}, \widetilde{X}^c)$ to 
  \begin{equation}\label{eq:MarkovianSDE_M'}
     \begin{cases}
         d\widetilde{X}_t = b(t, \widetilde{X}_t, \varphi_m(t, \widetilde{X}^c_t), \alpha_M(t, \widetilde{X}_t, \widetilde{X}^{c}_t))dt + \sigma d\widetilde{W}_t + \sigma^0 d\widetilde{W}^0_t\\
         d\widetilde{X}^c_t= b^c(t, \widetilde{X}^c_t)dt + \sigma^c d\widetilde{W}^0_t
 \end{cases}
 \end{equation}
 such that $\widetilde{\P} \circ (\widetilde{X}_t, \widetilde{X}^c_t)^{-1} = \P^{\alpha, m} \circ (X_t, X^c_t)^{-1}$ for all $t \in [0, T]$. 

Now define the bounded measurable drift function $b^{m, \alpha}: [0, T] \times \R^{d_I} \times \R^{d_C} \to \R^{d_I}$ as 
 \begin{equation}\label{eq:newdrift}
     b^{m, \alpha}(t, x, x^c) \ce b(t, x, \varphi_m(t, x^c), \alpha_M(t, x, x^c)).
 \end{equation}
 Applying Proposition \ref{proposition:MalliavinDiffBounded} to  this drift yields a unique strong solution $(X^{\alpha_M}, X^c)$ to \eqref{eq:MarkovianSDE_M'} on the original probability space. Uniqueness in law ensures that $\widetilde{\P}\circ (\widetilde{X}_t, \widetilde{X}^c_t)^{-1} = \P\circ(X^{\alpha_M}_t, X^c_t)^{-1}$ for all $t \in [0, T]$. 
 Similarly, define $\widetilde{\alpha} \in \A_w$ by $\widetilde{\alpha}_t \ce \alpha_M(t, X_t, X^c_t)$ where $X$ is again the driftless state process. Then by Girsanov theorem, under $\P^{\widetilde{\alpha}, m}$, the pair $(X, X^c)$ solves the Markovian SDE 
   \begin{equation}\label{eq:MarkovianSDE_M}
     \begin{cases}
         dX_t = b(t, X_t, \varphi_m(t, X^c_t), \alpha_M(t, X_t, X^{c}_t))dt + \sigma dW^{\widetilde{\alpha}, m}_t + \sigma^0 dW^0_t\\
         dX^c_t= b^c(t, X^c_t)dt + \sigma^c dW^0_t.
 \end{cases}
 \end{equation}
 Uniqueness in law of the SDE solution implies that for all $t \in [0, T]$,
 \begin{equation}\label{label:SameJointLaw}
     \P^{\widetilde{\alpha}, m} \circ (X_t, X^c_t)^{-1} = \P \circ (X^{\alpha_M}_t, X^c_t)^{-1} = \widetilde{\P} \circ (\widetilde{X}_t, \widetilde{X}^c_t)^{-1} = \P^{\alpha, m} \circ (X_t, X^c_t)^{-1}
 \end{equation} 
 for almost every $t$. Combining \eqref{label:SameJointLaw}, inequality \eqref{label:Mimicking2} and Fubini's theorem, we conclude:   
 \begin{align*}
         J^{m}_w(\widetilde{\alpha}) & = \E^{\P^{\widetilde{\alpha}, m}}\sqbra{g(X_T, m_T) + \intT f(t, X_t, m_t, \widetilde{\alpha}_t)dt}\\
         & = \E^{\P^{\tilde{\alpha}, m}}\sqbra{g(X_T, \varphi_m(T, X^c_T)) + \intT f(t, X_t, \varphi_m(t, X^c_t), \alpha_M(t, X_t, X^c_t))dt}\\
         & = \E^{\P^{\alpha, m}}\sqbra{g(X_T, \varphi_m(T, X^c_T)) + \intT f(t, X_t, \varphi_m(t, X^c_t), \alpha_M(t, X_t, X^c_t))dt}\\
         & \leq \E^{\P^{\alpha, m}}\sqbra{g(X_T, m_T) + \intT f(t, X_t, m_t, \alpha_t)dt}= J^{m}_w(\alpha).
 \end{align*}
 If $\alpha = \hat{\alpha}$, then this Markovian $\widetilde{\alpha}$ must also be optimal. From the previous step, recall that $(Y^m, Z^m, Z^{0, m}) = (Y^{m, \hat{\alpha}}, Z^{m, \hat{\alpha}}, Z^{0, m, \hat{\alpha}})$ uniquely solves \eqref{eq:optim.BSDE}. 
 For $\beta \in \A_w$, denote by $H^\beta$ the process $$H^\beta_t \ce H(t, X_t, m_t, \beta_t, Z^{m}_t).$$
 Suppose $\hat\alpha$ and $\widetilde{\alpha}$ differ on a set of positive $\P\otimes dt$ measure. The fact that $\hat{\alpha}$ uniquely minimizes the Hamiltonian implies that $H^{\hat{\alpha}} \leq H^{\widetilde{\alpha}}$ $~\P\otimes dt$-a.s., with strict inequality on a set with positive $\P\otimes dt$ measure. 
 By the strict comparison principle for Lipschitz BSDEs \cite[Theorem 2.2]{el1997backward}, this would imply $Y^{\hat{\alpha}}_0 < Y^{\widetilde{\alpha}}_0$ $\P$-almost surely, contradicting the optimality of $\widetilde{\alpha}$. Hence, $\hat{\alpha}$ and $\widetilde{\alpha}$ coincide $\P\otimes dt$-almost everywhere.
 \end{proof}

 Lemma \ref{Lemma:Mimicking} establishes that the optimal control $\hat{\alpha}$ has a Markovian representation $\widetilde{\alpha}$ which $\P^{\hat{\alpha}, m}\otimes dt$-almost everywhere coincides with $\hat{\alpha}$. Furthermore, we in fact have $\P^{\hat{\alpha}, m} = \P^{\widetilde{\alpha}, m}$ by \eqref{eq:Girsanov}. Recall that $(X^{\hat{\alpha}_M}, X^c)$ is the unique solution to \eqref{label:systemSDE} with drift $\tilde{b} = b_m$, as defined in \eqref{eq:newdrift}. By the uniqueness of law from Proposition \ref{proposition:MalliavinDiffBounded}, we obtain: $$\mm_\cdot \equiv \calL^{\P^{\hat{\alpha}, m}}\big(X_\cdot | X^c_\cdot) = \calL^{\P^{\tilde{\alpha}, m}}\big(X_\cdot | X^c_\cdot) = 
 \calL^{\P}\big(X^{\hat\alpha_M}_\cdot | X^c_\cdot).$$
Moreover, $(X^{\hat{\alpha}_M}, X^c) \in \X$, and so $\mm \in \overline{\M_{C, 1}} \subset \calK$.

 \emph{Step 3: The fixed point mapping $\Phi$ is continuous.} 
 Let $(m^n)_n$ be a sequence in $\calK$ converging to some $m^\infty$. For $n \in \N \cup \{\infty\}$, put $\mm^n = \Phi(m^n)$ and $\hat\alpha^n$ the optimal control associated with $m^n$. Define $M^n$ as the stochastic exponential for the change of measure in Girsanov's theorem:
 $$M^n_t \ce \mathcal{E}\pa{\int_0^\cdot \sigma^{-1}b(s, X_s, m^n_s, \hat{\alpha}^n_s)\,dW_s}_t.$$
 By Assumption \ref{Assumption:E1} and It\^{o}'s formula, there exists a uniform bound $C >0$ such that $\E\sqbra{|M^{n}_T|^4} \leq C$ for all $n \in \N \cup \{\infty\}$. Using the inequality $|e^{a} - e^{b}| \leq |a - b||e^{a} + e^{b}|$ and Cauchy-Schwarz inequality, we can find a (possibly different) $C > 0$ such that
 \begin{equation*}
     \E\sqbra{\abs{M^n_T - M^\infty_T}} \leq C\E\sqbra{\intT\abs{\sigma^{-1}\big (b(s, X_s, m^n_s, \hat{\alpha}^n_s) - b(s, X_s, m^{\infty}_s, \hat{\alpha}^\infty_s)\big)}^2ds}^{1/2}.
 \end{equation*}
 Recall that $\hat{\alpha}^n_t = \hat{a}(t, X_t, m^n_t, Z^{m^n}_t)$, where $Z^{m^n}$ is part of the unique solution to the Lipschitz BSDE \eqref{eq:optim.BSDE} with input $m^n$, and $\hat{a}$ is the minimizer function. Since $f$ and $g$ are assumed to be continuous in $m$, applying the stability result of Lipschitz BSDEs (see e.g. \cite[Proposition 2.1]{el1997backward}) yields that $Z^{m^n}$ converges to $Z^{m^\infty}$ in $\mathbb{H}^2$. 
 
Boundedness of $b$ and $b^c$ implies that for all $q \geq 1$: $$\sup_{n\geq1} \E^{\P^{\hat{\alpha}^n, m^n}}\Big[\sup_{t \in [0, T]}|X_t|^q\Big]<\infty, $$ 
which leads to uniform integrability, namely:
\begin{equation}\label{eq:uniform_integrability}
    \lim_{R \to \infty}\limsup_{n \to \infty} \Big(\E^{\P^{\hat{\alpha}^n, m^n}}+\E^{\P^{\hat{\alpha}^\infty, m^\infty}}\Big)\Big[\ind{\norm{X}_\infty > R}\sup_{t \in [0, T]}|X_t|^p\Big] = 0.
\end{equation}
This property establishes equivalence between convergence in $\calW_p$ and convergence in $\calW_1$. The ``unconditional'' version of this statement can be found, for example, in \cite[Theorem 6.9]{villani09}). We reformulate it under the conditional setup for completeness. 
\begin{lemma}\label{Lemma:WpisW1}
    Under \eqref{eq:uniform_integrability}, if $\E[\calW^q_q(\mm^n_t, \mm^\infty_t)] \nto 0$ holds for $q = 1$, then it holds for any $q \leq p$.
\end{lemma}
\begin{proof}
    It suffices to show the case for $q = p$.
    Define $c^+ \ce \max(c, 0)$ for $c \in \R$. For $R > 0$, observe that for any $(x, y) \in \R^{d_I} \times \R^{d_I}$:
    $$\ind{|x-y| \geq R} \leq \ind{|x| \geq R/2, |x| \geq |x-y|/2} + \ind{|y| \geq R/2, |y| \geq |x-y|/2}$$
    and therefore,
    \begin{equation}\label{label:indicator_Rp}
        (|x - y|^p - R^p)^+ \leq 2^p|x|^p\ind{|x| \geq R/2} + 2^p|y|^p\ind{|y| \geq R/2}.
    \end{equation}
    To simplify the notation, put $\mu^n_{x^c} =\calL^{\P^{\hat{\alpha}^n, m^n}}\big(X_t | X^c_t = x^c\big)$ for $n \in \N \cup \{\infty\}$. Fixing $X^c_t = x^c$, let $\pi_n(x^c)$ be an optimal transport plan for $\calW_1$ between $\mu^n_{x^c}$ and $\mu^\infty_{x^c}$, whose existence is guaranteed by \cite[Theorem 4.1]{villani09}. 
    By \eqref{label:indicator_Rp} and the definition of $\calW_p$, we have
    \begin{align*}
        \calW_p(\mu^n_{x^c}, \mu^\infty_{x^c}) & \leq \int (|x - y|^p \wedge R^p) + (|x - y|^p - R^p)^+ d\pi^n(x^c)(x, y)\\
        &  \leq \int (|x - y|^p \wedge R^p)d\pi^n(x^c)(x, y)\\
        & \quad + 2^p \E^{\mu^n_{x^c}}[|X_t|^p\ind{|X_t| > R/2}] + 2^p\E^{\mu^\infty_{x^c}}[|X_t|^p\ind{|X_t| > R/2}].
    \end{align*}
    Note that $(|x - y|^p \wedge R^p) \leq R^{p-1}|x-y|$. In particular, 
    \begin{equation*}
        \E_{x^c \sim \L^{\P}(X^c_t)}\bigg[\int (|x - y|^p \wedge R^p)d\pi^n(x^c)(x, y)\bigg] \leq R^{p-1}\E[\calW_1(\mm^n_t, \mm^\infty_t)]\nto 0.
    \end{equation*}
    Therefore, by law of total expectation and uniform integrability \eqref{eq:uniform_integrability}, we have
    \begin{align*}
        \limsup_{n \to \infty} \E[\calW_p(\mm^n_t, \mm^\infty_t)] \leq 2^{p}\lim_{R \to \infty}\limsup_{n \to \infty} \Big(\E^{\P^{\hat{\alpha}^n, m^n}}+\E^{\P^{\hat{\alpha}^\infty, m^\infty}}\Big)\Big[\ind{\norm{X}_\infty > R/2}\sup_{t \in [0, T]}|X_t|^p\Big] = 0.
    \end{align*}
\end{proof}

In light of this lemma, we will show that $\E[\calW_1(\mm^n_t, \mm^\infty_t)] \nto 0$. Given the continuity of $f$ and $b$ in $(a, \mu)$ and the compactness of $A$, Berge's maximum theorem (e.g. \cite[Theorem 17.31]{bookInfDimAnalysis06}) ensures that the minimizer function $\hat{a}$ is continuous in $(\mu, z)$. Consequently, $\hat{\alpha}^{n}$ also converges to $\hat{\alpha}^\infty$ in $\mathbb{H}^2$. Since $|\sigma^{-1}b|$ is bounded, applying bounded convergence theorem yields the convergence of $M^n_T$ to $M^\infty_T$ in $L^1(\Omega)$. Since $|M^n_T|^p$ is integrable uniformly in $n$, by dominated convergence theorem $M^n_T$ also converges to $M^{\infty}_T$ in $L^q(\Omega)$ for any $q \geq 1$. By Kantorovich-Rubinstein duality \eqref{eq:KantorovichDuality}:
 \begin{align*}
\calW_1\big(\mm^n_t, \mm^\infty_t \big) & = \sup_{\phi \in \text{Lip}_1(\R^{d_I})} \E^{\P^{\hat{\alpha}^n, m^n}}\Big[\phi(X_t)|X^c_t\Big] - \E^{\P^{\hat{\alpha}^\infty, m^\infty}}\Big[\phi(X_t)|X^c_t\Big]\\
     & =  \sup_{\phi \in \text{Lip}_1(\R^{d_I})}\frac{\E[M^n_T\phi(X_t)|X^c_t]}{\E[M^n_T|X^c_t]} - \frac{\E[M^\infty_T\phi(X_t)|X^c_t]}{\E[M^\infty_T|X^c_t]}.
 \end{align*}
 Note that $X^c_t$ is $\F^0_t$-measurable, which is independent from the stochastic exponential with respect to $W$. Then for all $k \in \N \cup \{\infty\}$ we have:
 $$\E[M^k_T|X^c_t] = \E[\E[M^k_T|\F_t]|X^c_t] = \E[M^k_t|X^c_t] = \E[M^k_t] = 1.$$
 Therefore, 
\begin{align*}
 \sup_{t \in [0, T]}\E\Big[\calW_1\big(\mm^n_t, \mm^\infty_t \big)\Big] & = \sup_{t \in [0, T]}\E\Big[\sup_{\phi \in \text{Lip}_1(\R^{d_I})} \E\Big[(M^n_T - M^{\infty}_T)\phi(X_t) \Big| X^c_t\Big]\Big]\\
      & \leq \sup_{t \in [0, T]}\E\Big[\E\Big[|M^n_T - M^{\infty}_T||X_t| \Big |X^c_t\Big]\Big]\\
     & \leq C\sup_{t\in [0, T]} \sqrt{\E\Big[|M^n_T - M^\infty_T|^2\Big]\E\Big[|X_t|^2\Big]} \nto 0.
\end{align*}
By Fubini's theorem, Lemma \ref{Lemma:WpisW1} and \eqref{eq:dMJensen}, the mapping $\Phi$ is continuous:
$$d^p_{\calM}(\mm^n, \mm^{\infty}) \leq T^{p/2 - 1} \intT \E\Big[\calW_p^p(\mm^n_t, \mm^{\infty}_t)\Big]dt \nto 0.$$

Finally, applying \cite[Corollary 17.56]{bookInfDimAnalysis06} yields a fixed point $\hat{m} \in \calK$ of the mapping $\Phi$. From Step 2, we know $\Phi(\calK) \subset \M_{C, 1}$, so the fixed point $\hat{m}$ belongs to $\M_{C,1}$, and the consistency condition is satisfied. This completes the proof of Theorem \ref{Theorem:MainWeak}.\qed
 ~\\

 Now we return to the strong formulation and Theorem \ref{Theorem:Main}. Let $(\hat{m}, \hat{\alpha}^w)$ be a MFG equilibrium under the weak formulation from Theorem \ref{Theorem:MainWeak}, and let $\hat{\alpha}_M$ be the corresponding measurable function from Lemma \ref{Lemma:Mimicking}. Then by the proof of Theorem \ref{Theorem:MainWeak}, we know $\hat{\alpha}^w_{\cdot} = \hat{\alpha}_M(\cdot, X_{\cdot}, X^c_{\cdot})$, $\P\otimes dt$ almost everywhere. Obtain $X^{\hat{\alpha}_M}$ and the corresponding Markovian control $\hat{\alpha}_{\cdot} \ce \hat{\alpha}_M(\cdot, X^{\hat{\alpha}_M}_{\cdot}, X^c_{\cdot}) \in \A$ by solving \eqref{eq:MarkovianSDE_M'}. Admissibility of $\hat{\alpha}_M$ is guaranteed by Proposition \ref{proposition:MalliavinDiffBounded}. We conclude the proof of Theorem \ref{Theorem:Main} with the following corollary. 

 \begin{corollary}
     $(\hat{m}, \hat{\alpha}) \in \M_{C, 1} \times \A$ is a MFG equilibrium under the strong formulation.
 \end{corollary}
 \begin{proof}
     Consistency of $\hat{m}$ still holds since for $t \in [0, T]$: 
     $$\hat{m}_t = \calL^{\P^{\hat{\alpha}^w, \hat{m}}}\big(X_t \mid X^c_t\big) = \calL^{\P}\big(X^{\hat{\alpha}_M}_t \mid X^c_t\big), \quad \P\as$$ 
     For optimality of $\hat{\alpha}$, recall $\varphi_{\hat{m}}$ from Lemma \ref{lemma:calK}. By \eqref{label:SameJointLaw} we have
     \begin{align*}
         &J^{\hat{m}}_w(\hat{\alpha}^w)  = \E^{\P^{\hat{\alpha}, \hat{m}}}\sqbra{g(X_T, \varphi_{\hat{m}}(T, X^c_T)) + \intT f(t, X_t, \varphi_{\hat{m}}(t, X^c_t), \hat\alpha_M(t, X_t, X^c_t))dt}\\
         & = \E\sqbra{g(X^{\hat{\alpha}_M}_T,  \varphi_{\hat{m}}(T, X^c_T)) + \intT f(t, X^{\hat{\alpha}_M}_t,  \varphi_{\hat{m}}(t, X^c_t), \hat\alpha_M(t, X^{\hat{\alpha}_M}_t, X^c_t))dt} =  J^{\hat{m}}(\hat{\alpha}).
     \end{align*}
    
     For any other admissible strategy $\alpha \in \A$ such that \eqref{eq:Strong_State} has a unique strong solution $X^{\alpha}$, the proof of Lemma \ref{Lemma:Mimicking} implies the existence of a measurable function $\alpha_M:[0, T] \times \R^{d_I} \times \R^{d_C}\to A$ such that $\P\circ (X^{\alpha}_t, X^c_t)^{-1} = \P\circ (X^{\alpha_M}_t, X^c_t)^{-1} =  \P^{\widetilde{\alpha}, \hat{m}} \circ (X_t, X^c_t)^{-1}$ for all $t \in [0, T]$, where  $\widetilde{\alpha}_{\cdot} \ce \alpha_M(\cdot, X_{\cdot}, X^c_{\cdot})$. Moreover 
     $$f(t, X^{\alpha_M}_t, \hat{m}_t, \alpha_M(t, X^{\alpha_M}_t, X^c_t)) \leq \E\sqbra{f(t, X^{\alpha}_t, \hat{m}_t, \alpha_t) \mid X^{\alpha}_t, X^c_t}, \text{ for } t \in [0, T].$$
     By optimality of $\hat{\alpha}^w$ in the weak formulation, we have 
     \begin{align*}
         J^{\hat{m}}(\alpha) & = \E\sqbra{g(X^{\alpha}_T, \hat{m}_T) + \intT f(t, X^{\alpha}_t,  \hat{m}_t, \alpha_t )dt}\\
         & \geq \E\sqbra{g(X^{\alpha_M}_T,  \varphi_{\hat{m}}(T, X^c_T)) + \intT f(t, X^{\alpha_M}_t,  \varphi_{\hat{m}}(t, X^c_t), \alpha_M(t, X^{\alpha_M}_t, X^c_t))dt}\\
         & =\E^{\P^{\widetilde{\alpha}, \hat{m}}}\sqbra{g(X_T,  \varphi_{\hat{m}}(T, X^c_T)) + \intT f(t, X_t,  \varphi_{\hat{m}}(t, X^c_t), \alpha_M(t, X_t, X^c_t))dt}\\
         & = J_w^{\hat{m}}(\widetilde{\alpha}) \geq J^{\hat{m}}_w(\hat{\alpha}^w) = J^{\hat{m}}(\hat\alpha).
     \end{align*}
 Thus, $(\hat{m}, \hat{\alpha})$ is a MFG equilibrium for the strong formulation.
 \end{proof}

\begin{corollary}
\label{cor:n.steps.mfg}
    Fix a partition $(0 = t_0\leq t_1 \leq \dots \leq t_n = T)$ of $[0,T]$ and call mean field equilibrium a pair $(\hat m,\hat\alpha) \in \calM \times \A$ such that $\hat\alpha \in \A$ minimizes the objective $J^{\hat{m}}$ defined in \eqref{eq:Strong_Objective} over $\A$ and for every $t \in [0, T]$, $\hat m_t$ satisfies
 \begin{equation*}
     \hat m_t = \calL^{\P}\big(X^{\hat\alpha}_t | X^c_{t \wedge t_0}, \cdots, X^c_{t \wedge t_n}\big), \quad \P\as \text{ for all } t \in [0, T].
 \end{equation*} 
    If Assumptions \ref{Assumption:E} and \ref{Assumption:M} hold, then a mean field equilibrium exists. 
\end{corollary}
\begin{proof}
    To simplify the discussion assume $n=2$.
    Put $\mathbb{X}^c_t := (X^1_t,X^2_t)$ with
\begin{align*}
    X^i_t &= X^c_{t\wedge t_i} = \xi^c + \int_0^tb^c(s,X^i_s)1_{\{s\le t_i\}}ds + \int_0^t\sigma^c1_{\{s\le t_i\}}dW_s^0.
\end{align*}
Then we have
$\hat m_t = \mathcal{L}^{\mathbb{P}}(X^{\hat{\alpha}}_t|\mathbb{X}_t^c)$ with $\mathbb{X}^c$ satisfying $d\mathbb{X}^c_t = B^c(t,\mathbb{X}^c_t)dt + \Sigma_td\mathbb{W}^c_t$, $\mathbb{X}_0 = (\xi^c,\xi^c)$, $\mathbb{W} = (W^0,W^0)$
with appropriately defined vector $B^c$ and diagonal matrix $\Sigma$. 
We would be exactly in the setting of Theorem \ref{thm:main} (with newe common state $\mathbb{X}^c$) if not for the fact that the diffusion coefficient is now time-dependent.
Despite, noticing that it is piece-wise constant allows us to repeat the proof of Theorem \ref{thm:main}, especially the crucial Proposition \ref{proposition:MalliavinDiffBounded} which is the only step where the constant volatility was used.

In fact, for any uniformly bounded, Borel-measurable functions $\tilde{b}: [0, T] \times \R^{d_I} \times (\R^{d_C})^2 \to  \R^{d_I}$, the following SDE has a unique strong solution in $\bL_{1,2}^a(\R^{d_I} \times (\R^{d_C})^2, \mathbb{F})$:
     \begin{equation}\label{label:systemSDE.ndim}
         \begin{cases}
         d\widetilde{X}_t = \tilde{b}(t, \widetilde{X}_t, \mathbb{X}^c_t)dt + \sigma dW_t + \sigma^0 dW^0_t, \quad \widetilde{X}_0 = \xi\\
         d\mathbb{X}^c_t= B^c(t, \mathbb{X}^c_t)dt + \Sigma_t d\mathbb{W}^0_t, \quad \mathbb{X}^c_0 = (\xi^c,\xi^c).
     \end{cases}
     \end{equation}
     This is because for the time horizon $t = t_1$, the result follows by Proposition \ref{proposition:MalliavinDiffBounded} and on the time interval $[t_1,T]$, given that the new initial condition $\mathbb{X}^c_{t_1}$ is Malliavin differentiable, we can apply again Proposition \ref{proposition:MalliavinDiffBounded} and pasting both solutions on smaller time intervals gives the unique solution of the SDE on $[0,T]$.
     To show that the set $\X$ of all pairs $(\widetilde{X}, \mathbb{X}^c)$ that (uniquely) solve the system of SDEs \eqref{label:systemSDE.ndim} for some measurable, bounded $\tilde{b}$ with bound $\lVert\tilde{b}\rVert_\infty \leq \norm{b}_\infty$ is compact, as in the proof of Proposition \ref{proposition:MalliavinDiffBounded} it suffices to find constant $C >0$ such that for $M:=(X,\mathbb{X}^c)$, we have
     $\E\sqbra{|M_t|^2}\le C$ for all $t \in [0, T]$ and
     \begin{equation*}
     \E[\norm{D_sM_t - D_{s'}M_t}^2] \leq C|s - s'|, \text{ and } \sup_{0 \leq u \leq t} \E[\norm{D_uM_t}^2]\leq C.
     \end{equation*}    
     for all $0 \leq s, s' \leq t \leq T$.
     For $t\le t_1$ this property directly follows from Proposition \ref{proposition:MalliavinDiffBounded} because $\mathbb{X}^c_t = (X^c_t, X^c_t)$.
     For $t\ge t_1$, if $t_1 \leq s, s' \leq T$, we have $\mathbb{X}^c_t = (X^c_{t_1}, X^c_{t})$, so the property follows again. If $s \leq t_1 \leq s' \leq t$, triangle inequality leads to the same criterion. 
     Notice however that when $n$ is arbitrary, the constant $C$ will depend on $n$.
     With this key compactness property at hand, the rest of the proof of existence of a mean field equilibrium is the same, replacing $X^c$ by $\mathbb {X}^c$ and $d_c$ by $d_c\times n$.
\end{proof}

\begin{remark}\label{remark:generalMFG}
We contrast our consistency condition in Definition \ref{Def:Strong_MFGsolution} with standard formulation of MFGs with common noise, given by:
     $\hat m_t = \calL^{\P}\big(X^{\hat\alpha}_t | \F^0_t\big)$,  $\P\as \text{ for all } t \in [0, T]$.
 Our method does not directly extend to this setting because it would require simultaneously mimicking the new process $m$, which alters the process and compromises the optimality of the control.
In addition, this will require existence and compactness estimates (in the form of Proposition \ref{proposition:MalliavinDiffBounded}) for non-Markovian SDEs.
 However, as observed in Corollary \ref{cor:n.steps.mfg}, 
 the consistency condition can be enforced at finitely many time points. 
Coming back to Remark \ref{rem:def.equil}, this shows that our setting is consistent with general finite particle models in which players ``react'' to the common noise at finitely many time points, or where the common noise is a discrete process.

It is tempting to take $n\to\infty$ to recover the full past-dependent case. 
The resulting equilibrium would be weak, which is due to the dependence of the constant $C$ in Proposition \ref{proposition:MalliavinDiffBounded} on the dimensions, which would explode as $n \to \infty$. Nonetheless, this method still partially generalizes the approach in \cite{CarmonaCommonNoise}, as we discretize the common noise in time only, but not in space.
\end{remark}


\end{document}